\documentclass[a4paper, 10pt]{article}
\usepackage{amsmath}
\usepackage{amsthm}
\usepackage{amssymb}
\usepackage{latexsym}
\usepackage{enumerate}

\newtheorem{thm}{Theorem}[section]
\newtheorem{prop}{Proposition}[section]
\newtheorem{lem}{Lemma}[section]
\newtheorem{cor}{Corollary}[section]

\title{The lower bound of the Ricci curvature that yields the infinite number of the discrete spectrum of the Laplacian}

\author{Hironori Kumura}

\date{}

\begin{document}

\maketitle

\begin{abstract}
This paper discusses the question whether the discrete spectrum of the Laplace-Beltrami operator is infinite or finite. 
The borderline-behavior of the curvatures for this problem will be completely determined. 
\end{abstract}
\section{Introduction}

The Laplace-Beltrami operator $\Delta$ on a noncompact complete Riemannian manifold $(M,g)$ is essentially self-adjoit on $C^{\infty}_0(M)$ and its self-adjoit extension to $L^2(M)$ has been studied by several authors from various points of view. 
In many cases, the bottom of the essential spectrum of $-\Delta$ will be positive (see Brooks \cite{B}), and the discrete spectrum will appear below this bottom number. 
The purpose of this paper is to determine the borderline-behavior of curvatures for the question whether the Laplace-Beltrami operator $- \Delta $ has a finite or infinite number of the discrete spectrum. 
The Rellich's lemma (see, for example, M.~Taylor \cite{T} ) suggests that this problem depends on the geometry of manifolds {\it at infinity}. 
In the case of Schr\"odinger operators $-\Delta +V$ on the Euclidean space ${\bf R}^n$, the borderline-behavior $-\frac{(n-2)^2}{4r^2}$ of the potential $V$ is determined by the {\it uncertainty principle lemma} 
$-\Delta \ge \frac{(n-2)^2}{4r^2}$ (see Reed-Simon \cite{R-S II} pp. 169 and Kirsh-S \cite{K-S} ), which is equivalent to the Hardy's inequality $-\frac{d^2u}{dx^2}\ge \frac{1}{4x^2}$ for $u\in C_0^{\infty}(0,\infty)$ (see, for example, \cite{A-K}). 
Our proof will be concerned with this borderline-behavior of the Hardy's inequality (see Proposition $2.1$ in section $2$). 

Main theorems of this paper is the following:
%
%
\begin{thm}
Let $(M,g)$ be an $n$-dimensional noncompact complete Riemannian manifold and $W$ a relatively compact open subset of $M$ with $C^{\infty}$-boundary $\partial W$. 
We set $r(*):={\rm dist}(*,\partial W)$ on $M\backslash W$. 
Let $\exp_{\partial W}: \mathcal{N}^+(\partial W) \rightarrow M\backslash W$ be the outward normal exponential map and ${\mathcal Cut}(\partial W)$ the corresponding cut locus of $\partial W$ in $M \backslash W$, where 
\begin{align*} 
  \mathcal{N}^+(\partial W) 
:= 
  \big\{ v \in TM|_{\partial W} \mid v \ {\rm is\ outward\ normal\ to}\ \partial W \big\}.
\end{align*}
Assume that  
\begin{align*}
 \min \sigma_{{\rm ess}} (-\Delta ) = \frac{(n-1)^2 \kappa }{4}
\end{align*}
for some constant $\kappa >0$ and that there exist positive constants $R_0$ and $\beta$, satisfying $\beta > \frac{1}{(n-1)^2}$, such that
\begin{align*}
 & {\rm Ric}_{g}\left( \nabla r, \nabla r \right) (y)
   \ge 
   (n-1)\left( - \kappa + \frac{\beta}{r(y)^2} \right) \\
 & \hspace{30mm}\mbox{for}~y \in M \backslash 
   \left( W \cup {\mathcal Cut}(\partial W) \right)~\mbox{with}~ r(y) \ge R_0,
\end{align*}
where ${\rm Ric}_{g}$ and $\nabla r$ respectively stand for the Ricci curvature of $(M,g)$ and the gradient of the function $r$. 
Then, the set
\begin{align*}
 \sigma_{\rm disc}( -\Delta ) \cap \left[ 0, \frac{(n-1)^2 \kappa }{4} \right)
\end{align*}
is infinite, where $\sigma_{\rm disc}( -\Delta )$ stands for the discrete spectrum of $-\Delta$. 
\end{thm}
Note that we do not assume that $M\backslash W$ is connected in Theorem $1.1$: hence, $\partial W$ may have several but finite number of components. 

Similarly, we get the following:
\begin{thm}
Let $(M, g)$ be an $n$-dimensional complete Riemannian manifold and $p_0$ be a point of $M$. 
We set $r(*):={\rm dist}(*,p_0)$ and denote by ${\mathcal Cut}(p_0)$ the cut locus of $p_0$. 
Assume that  
\begin{align*}
 \min \sigma_{{\rm ess}} (-\Delta ) = \frac{(n-1)^2 \kappa }{4}
\end{align*}
for some constant $\kappa >0$ and that there exist positive constants $R_0$ and $\beta$, satisfying $\beta > \frac{1}{(n-1)^2}$, such that
\begin{align*}
 & {\rm Ric}_{g}(\nabla r, \nabla r) (y)
   \ge 
   (n-1)\left( -\kappa + \frac{\beta}{r(y)^2} \right) \\
 & \hspace{30mm} \mbox{for}~y \in 
   M \backslash {\mathcal Cut}(p_0)~\mbox{with}~ r(y) \ge R_0.
\end{align*}
Then, the set
\begin{align*}
 \sigma_{\rm disc}( -\Delta ) \cap \left[ 0, \frac{(n-1)^2 \kappa }{4} \right)
\end{align*}
is infinite, where $\sigma_{\rm disc}( -\Delta )$ stands for the discrete spectrum of $-\Delta$. 
\end{thm}
Although the topological property of manifolds is reflected in that of the cut locus, the theorem above does {\it not concern the property of the cut locus at all} but only the {\it Ricci curvatures of the radial direction on the complement of the cut locus}. 

The following proposition shows that the curvature assumption in Theorem $1.1$ and $1.2$ are sharp:
\begin{prop}
Let $\left({\bf R}^n, dr^2 + h^2(r) g_{S^{n-1}(1)} \right)$ be a rotationally symmetric Riemannian manifold and assume that the radial curvature $K(r) = -\frac{h''(r)}{h(r)}$ satisfies
\begin{align*}
  K(r) \le 0  \qquad \mbox{for~all}~~ r\ge 0
\end{align*}
and there exists constants $\kappa >0$, $R_0 >0 $ and $\beta \neq \frac{1}{(n-1)^2}$ such that
\begin{align*}
  K(r) = - \kappa + \frac{\beta}{r^2} \qquad \mbox{for}~~ r \ge R_0.
\end{align*}
Then, $\sigma_{{\rm ess}}(-\Delta ) = \left[\frac{(n-1)^2\kappa}{4},\infty \right)$, and furthermore, $\sigma_{\rm disc}( -\Delta ) \cap \left[ 0, \frac{(n-1)^2 \kappa }{4} \right)$ is infinite if and only if $\beta > \frac{1}{(n-1)^2}$. 
\end{prop}

Indeed, under the assumptions in Proposition $1.1$, 
${\rm Ric}_{g}(\nabla r, \nabla r) = (n-1) K(r) = (n-1)\left( - \kappa + \frac{\beta}{r^2} \right)$, and hence, the lower bound of the Ricci curvature in Theorem $1.1$ and $1.2$ are sharp. 
That is, the borderline-behavior of curvatures for our problem can be said to be $-\kappa + \frac{1}{\{(n-1)r\}^2}$. 
See also \cite{A-K} Theorem $3.1$ for the finiteness-result on not necessarily rotationally symmetric manifolds. 

\section{Construction of a model space and eigenfunction}

In this section, we shall construct a model space and study the property of 
an eigenfunction, which will be transplanted on $M$ to prove Theorem $1.1$. 

Let $R_{\min} :[0,\infty) \to (-\infty, 0]$ be a nonpositive-valued continuous function satisfying
\begin{align*}
  {\rm Ric}_g \, (\nabla r, \nabla r) (x)
 \ge 
  (n-1) R_{\min} \left(r(x)\right)\quad \mbox{for}~~
  x\in M \backslash {\mathcal Cut}(p_0)
\end{align*}
and 
\begin{align}
  R_{\min} (r)= - \kappa + \frac{\beta}{r^2} \quad {\rm for}~~r \ge R_1,
\end{align}
where $\kappa >0$ and $R_1>R_0$ are constants. 

Using this function $R_{\min}(t)$, consider the solution $J(t)$ to the following classical Jacobi equation:
\begin{align*}
  J''(t) + R_{\min} (t) J(t) = 0;~~J(0)=0;~~J'(0)=1
\end{align*}
and set
$$
   S(t)=\frac{J'(t)}{J(t)}.
$$

Using this function $J$, let us consider a model space:
\begin{align*}
  M_{{\rm model}} := ({\bf R}^n , dr^2 + J(r)^2 g_{S^{n-1}(1)}),
\end{align*}
where $r$ is the Euclidean distance to the origin and $g_{S^{n-1}(1)}$ stands for the standard metric on the unit sphere $S^{n-1}(1)$. 

Since $\lim_{t \to +0}S(t)=\infty$, the Laplacian comparison theorem (see Kasue \cite{Ka}) implies that
\begin{align}
  \Delta r = \Delta_{(M,g)} \, r \le (n-1)S(r) 
  \qquad {\rm on}\quad M\backslash (W \cup {\mathcal Cut}(\partial W)). 
\end{align}
This inequality $(2)$ is known to hold on $M\backslash W$ in the sense of distribution. 
Note that $J(t) \ge t >0$ due to the non-positivity of $R_{\min} $, and hence, $S(t)=\frac{J'(t)}{J(t)}$ exists for all $t \in (0,\infty)$. 

Since $S(t) = \frac{J'(t)}{J(t)}$ satisfies the Riccati equation
\begin{align}
  S'(t) + S^2(t) + R_{\min} (t) = 0
\end{align}
and $R_{\min} (t)$ satisfies $(1)$, it is not hard to see that the solution $S(t)$ to this equation $(3)$ has the asymptotic behavior
\begin{align}
  S(t) = \sqrt{\kappa}-\frac{\beta}{2\sqrt{\kappa}\,t^2} + O \left( \frac{1}{t^3} \right).
\end{align}

The following proposition, which also plays an important role in {\rm \cite{A-K}}, serves to construct an eigenfunction on our model space $M_{{\rm model}}$: 

\begin{prop}
For any $R>0$ and $\delta >0$, consider the following eigenvalue problem $(*)$$:$
\begin{align*}
(*)
\begin{cases} 
   - \varphi ''(x) - (1 + \delta )\displaystyle \frac{1}{4 x^2} \varphi (x)
    = \lambda \, \varphi (x)\quad {\rm on}~~[R, 2kR];\\
   \varphi (R) = \varphi (2kR) = 0.
\end{cases}
\end{align*}
Then, the first eigenvalue $ - \lambda _1 = - \lambda _1 (\delta, R, k)$ of this problem $(*)$ is negative, if $k > 2 \left\{ \exp \left( \frac{12}{\delta} \right) \wedge 1 \right\}$. 
Here, we write $\exp \left( \frac{12}{\delta} \right) \wedge 1 = \min \left\{\exp \left( \frac{12}{\delta} \right),1 \right\}$. 
\end{prop}
\begin{proof}
We set 
$$ 
\chi(x) := 
\begin{cases} 
\ \ \frac{1}{R}(x - R)\quad & {\rm if}\quad x \in [R, 2R], \\ 
\ \ 1\quad & {\rm if}\quad x \in [2R, kR], \\ 
\ \ - \frac{1}{kR}(x - 2kR)\quad & {\rm if}\quad x \in [kR, 2kR], \\ 
\end{cases} 
$$ 
where $k > 2$ is a large positive constant defined later. 
Set $\varphi(x) := \chi(x) x^{\frac{1}{2}}$. 
Then, the direct computation shows that 
$$ 
|\varphi'(x)|^2 - (1 + \delta) \displaystyle \frac{1}{4 x^2} |\varphi(x)|^2 
= |\chi'(x)|^2 x -  \frac{\delta}{4 x^2}  |\varphi(x)|^2 + 
\frac{1}{2} \left(\chi(x)^2\right)'. 
$$ 
Integrating the both sides over $[R,2kR]$, we have 
\begin{align*} 
&\int_R^{2kR} 
\left\{ |\varphi'|^2 - (1 + \delta )\frac{1}{4 x^2} |\varphi|^2 \right\} dx\\ 
= & \int_R^{2kR} |\chi'(x)|^2 x \, dx 
- \frac{\delta}{4} \int_R^{2kR} \frac{\chi^2(x)}{x} \, dx \\ 
\le & \frac{1}{R^2} \int_R^{2R} x \, dx + 
\frac{1}{(kR)^2} \int_{kR}^{2kR} x \, dx 
- \frac{\delta}{4} \int_{2R}^{kR} \frac{\chi^2(x)}{x} \, dx \\ 
= & 3 - \frac{\delta}{4} \log \left( \frac{k}{2} \right). 
\end{align*} 
Hence,  
$$ 
\int_R^{2kR} 
\left\{ |\varphi'|^2 - (1 + \delta )\frac{1}{4 x^2} |\varphi|^2 \right\} dx < 0
\qquad {\rm if}\quad k > 2 \left\{ \exp \left( \frac{12}{\delta} \right)\wedge 1\right\}. 
$$ 
Therefore, mini-max principle implies that the first eigenvalue of the problem $(*)$ is negative, if $k > 2 \left\{ \exp \left( \frac{12}{\delta} \right) \wedge 1 \right\}$.
\end{proof}

For $t>0$, we denote by $B(t)_{M_{\rm model}}$ the open ball of $M_{\rm model}$ centered at the origin $0$ with radius $t$, and by $\lambda _D \bigl( B(t)_{M_{\rm model}} \bigr)$  the first Dirichlet eigenvalue of $-\Delta_{M_{\rm model}}$ on $B(t)_{M_{\rm model}}$. 
Then, we have the following:
\begin{prop}
Assume that $\beta (n-1)^2 > 1$ and choose small constant $\delta >0$ so that $\beta (n-1)^2 > 1 + \delta$. 
For a fixed constant $k > 2 \left\{ \exp \left( \frac{12}{\delta} \right) \wedge 1 \right\}$, let $ - \lambda_1 = - \lambda_1(k, R, \delta)< 0$ be the first Dirichlet eigenvalue of the problem $(*)$. 
Then, there exists a positive constant $R_0(n,\beta,\kappa,\delta,R_{{\rm min}})$ such that 
\begin{align}
 \lambda _D \bigl( B(2kR)_{M_{\rm model}} \bigr)
 <
 \frac{(n-1)^2 \kappa}{4} - \lambda_1 
\end{align}
holds for any $R \ge R_0(n,\beta,\kappa,\delta, R_{{\rm min}})$. 
\end{prop}
\begin{proof}
Let $\varphi_1(x)$ be an eigenfunction of the problem $(*)$ with the first Dirichlet eigenvalue $- \lambda_1(k, R, \delta)< 0$.  
Then, we have 
\begin{align}
  \int_R^{2kR} |\varphi_1'(x)|^2 \, dx = 
  (1 + \delta ) \int_R^{2kR} \frac{1}{4x^2} |\varphi_1(x)|^2 \, dx
   -\lambda_1 \int_R^{2kR} |\varphi_1(x)|^2 \, dx .
\end{align}
We set 
$$ 
   f(x)
   =\varphi_1(x) 
   J^{-\frac{n-1}{2}}(x).
$$ 
Then, direct computations show that
\begin{align*}
  f'(x)
   =J^{-\frac{n-1}{2}}(x)
   \left\{ \varphi_1'(x) - \frac{n-1}{2} S(x) \varphi_1(x) \right\}
\end{align*}
and 
\begin{align*}
  &| f'(x) |^2 J^{(n-1)}(x)\\
= & | \varphi_1'(x) |^2 + 
  \frac{(n-1)^2}{4} S^2(x) | \varphi_1(x) |^2 - 
  \frac{n-1}{2} S(x) \left\{ \varphi_1(x)^2 \right\} ' .
\end{align*}
As for the last term 
$ - \frac{n-1}{2} S(x) \left\{ \varphi_1(x)^2 \right\} ' $, we calculate 
\begin{align*}
  - \frac{n-1}{2} \int_R^{2kR} 
  S(x) \left\{ \varphi_1(x)^2 \right\} ' \, dx
  = \frac{n-1}{2} \int_R^{2kR} 
  S'(x) | \varphi_1(x) |^2 \, dx,
\end{align*}
and hence, 
\begin{align*}
  & \int_R^{2kR} | f'(x) |^2 J^{n-1}(x) \, dx \\
= & \int_R^{2kR} \left\{ | \varphi_1'(x) |^2 + 
  \frac{n-1}{2} \left( \frac{n-1}{2} S^2(x) 
  + S'(x) \right) | \varphi_1(x) |^2 \right\} \, dx \\
= & \int_R^{2kR} \left\{ | \varphi_1'(x) |^2 + 
  \frac{n-1}{2} \left( \frac{n-3}{2} S^2(x) 
  - R_{\min}(x) \right) | \varphi_1(x) |^2 \right\} \, dx \\
= & \int_R^{2kR} \left\{ \frac{1+\delta}{4x^2} - \lambda_1 + 
  \frac{n-1}{2} \left( \frac{n-3}{2} S^2(x) 
  - R_{\min}(x) \right) \right\} | \varphi_1(x) |^2 \, dx ,
\end{align*}
where we have used equations $(3)$  and $(6)$. 
Here, by $(1)$ and $(4)$, 
\begin{align*}
  & \frac{n-1}{2} \left( \frac{n-3}{2} S^2(x) - R_{\min}(x) \right)\\
= & \frac{n-1}{2} \left\{ \frac{n-3}{2} 
  \left( \sqrt{\kappa} - \frac{\beta}{2\sqrt{\kappa}\,x^2} + 
  O\left(\frac{1}{x^3}\right) \right)^2 + 
  \kappa - \frac{\beta }{x^2} \right\}\\
= & \frac{(n-1)^2 \kappa}{4} - \frac{\beta (n-1)^2}{4x^2} + 
  O\left(\frac{1}{x^3}\right)
\end{align*}
and, therefore,
\begin{align*}
  & \int_R^{2kR} | f'(x) |^2 J^{n-1}(x) \, dx \\
= & \int_R^{2kR} \left\{ \frac{(n-1)^2 \kappa}{4} - \lambda_1 - 
    \frac{1}{4x^2}\left( \beta (n-1)^2 - 1 - \delta \right) 
    + O\left(\frac{1}{x^3}\right) \right\} | \varphi_1(x) |^2 \, dx.
\end{align*}
Since $\beta (n-1)^2 - 1 - \delta >0$ and 
$| \varphi_1(x) |^2=|f(x)|^2 J^{n-1}(x)$, we see that
\begin{align}
  \int_R^{2kR} | f'(x) |^2 J^{n-1}(x) \, dx 
  <
  \left( \frac{(n-1)^2 \kappa}{4} - \lambda_1 \right) 
  \int_R^{2kR} | f(x) |^2 J^{n-1}(x) \, dx 
\end{align}
for $R \ge R_0(n,\beta,\kappa,\delta,R_{{\rm min}})$.

Now, for $y\in M_{\rm model}$, we set 
$$
  \phi(y) := 
\begin{cases} 
\ \ f(r(y)), \quad & {\rm if}\quad r(y) \in [R, 2kR], \\ 
\ \ 0, \quad & {\rm otherwise}. \\ 
\end{cases}
$$
Then, integrating $(7)$ over $S^{n-1}(1)$ with its standard measure, we have
$$
   \int_{M_{\rm model}} |\nabla \phi |^2 dv_{M_{\rm model}}
   < 
   \left( \frac{(n-1)^2 \kappa}{4} - \lambda_1 \right) 
   \int_{M_{\rm model}} | \phi |^2 dv_{M_{\rm model}}.
$$
Hence, mini-max principle implies our desired inequality $(5)$ for \linebreak 
$R \ge R_0(n,\beta,\kappa,\delta,R_{{\rm min}})$. 
\end{proof}
Let $\psi _1$ denote the first Dirichlet eigenfunction of ball $B(2kR)_{M_{\rm model}}$ for $R \ge R_0(n,\beta,\kappa,\delta,R_{{\rm min}})$. 
Then, $\psi _1$ is radial, that is, 
\begin{align}
  \psi _1 (y)= h_1\left(r(y)\right)
\end{align}
for some function $h_1:[0,2kR]\to {\bf R}$ and $h_1$ satisfies the equation
\begin{align}
  -h_1''(x) - (n-1) S(x) h_1'(x) 
  = \lambda _D \bigl( B(2kR)_{M_{\rm model}} \bigr) h_1(x)
\end{align}
on the interval $(0,2kR]$. 
Since $h_1$ takes the same sign on $[0,2kR)$ (by maximum principle, or see Pr\"ufer \cite{Pr}), we may assume that 
\begin{align}
  h_1(x) > 0 \qquad {\rm on}~~[0,2kR).
\end{align}
Here, we claim the following crucial fact for our proof: 
\begin{lem}
  Under the assumption $(10)$, $h_1$ satisfies
  \begin{align}
     h_1'(x) < 0\qquad {\rm on}~~(0,2kR].
  \end{align}
\end{lem}
\begin{proof}
The proof is by contradiction. 
 
First, let us assume that $h_1'(2kR)=0$. 
Then, since $h_1$ satisfies $(9)$ and $h_1(2kR)=0$, $h_1(x)\equiv 0$ which contradict our assumption $(10)$. 
Therefore, we see that $h_1'(2kR) < 0$ by $(10)$ and $h_1(2kR)=0$. 

Next, let us assume that $h_1'(x_0) > 0$ for some $x_0 \in (0,2kR)$. 
Then, $h_1$ must takes a minimal value at a point, say $x_1$, in $[0,x_0)$. 
If $x_1 \in (0,x_0)$, 
\begin{align}
 - h_1''(x_1) 
 = \lambda _D \bigl( B(2kR)_{M_{\rm model}} \bigr) h_1(x_1) 
 > 0
\end{align}
by our assumption $(10)$. 
However, this contradicts our assumption that $h_1$ takes a minimal value at $x_1$. 
Therefore, $x_1=0$. 
Since $h_1'(0)=0$, $f(0)=0$, $f'(0)=1$, and $S(x)=\frac{f'(x)}{f(x)}$, we see that 
\begin{align*}
  \lim _{x\to +0} S(x) h_1'(x) = h_1''(0),
\end{align*}
and hence, by $(9)$, 
\begin{align}
 - n h_1''(0) 
 = 
 \lambda _D \bigl( B(2kR)_{M_{\rm model}} \bigr) h_1(0)
 > 0.
\end{align}
Two equations $h_1'(0)=0$ and $(13)$ imply that $0$ is a maximal point of $h_1$. 
However, this contradicts our assertion, proved above, that $x_1=0$ is a minimal point of $h_1$. 
 
Thus, we have proved that 
$$
    h_1'(x)\le 0 \qquad  {\rm on}~~(0,2kR).
$$
However, if $h_1'(x_2)=0$ for some $x_2 \in (0,2kR)$, $x_2$ must be a maximal point of $h_1$ by the same reason as is seen in $(12)$. 
Therefore, $h_1'(x_2 - \varepsilon ) > 0$ for small $\varepsilon > 0$. 
This also leads to a contradiction as is seen above. 
Thus, we have proved $(11)$. 
\end{proof}
\section{Proof of Theorem $1.1$ and $1.2$}
Let us start with notations involving the cut locus ${\mathcal Cut}(\partial W)$ of the boundary $\partial W$ in $M\backslash W$. 
Assume that $W$ be a relatively compact open subset of $M$ with $C^{\infty}$-boundary $\partial W$ and let $\exp_{\partial W}:\mathcal{N}^+(\partial W)\to M\backslash W$ be the outward exponential map. 
Let $\overrightarrow{n}$ be the outward unit normal vector field along $\partial W$ and set
\begin{align*}
 & U \mathcal{N}^+(\partial W) = \{ v \in \mathcal{N}^+(\partial W) 
   \mid |v|=1 \}, \\
 & {\mathcal B}(\partial W, \delta) = \{ v \in \mathcal{N}^+(\partial W) 
   \mid |v| < \delta \}, \\
 & B(\partial W, \delta) 
   = \{ y \in M\backslash W \mid {\rm dist}\,(W,y) < \delta \}.
\end{align*}
Moreover, for each $v \in U \mathcal{N}^+(\partial W)$, define
\begin{align*}
  \rho (v) 
  = \sup 
  \left\{ t > 0 \mid {\rm dist}\,\left(W,\exp_{\partial W} (tv) \right)=t \right\}
\end{align*}
and
\begin{align*}
  {\mathcal D}_{\partial W}
  = \{ tv \in \mathcal{N}^+(\partial W) \mid 0 \le t < \rho (v), \,
  v \in U \mathcal{N}^+(\partial W) \}.
\end{align*}
Then, ${\mathcal Cut}(\partial W) = \left\{ \exp_{\partial W}\left( \rho(v)v \right) \mid v \in U \mathcal{N}^+(\partial W) \right\}$. 
Let $dA$ denote the induced measure on the boundary $\partial W$ and write the Riemannian measure $dv_g$ on the domain $\exp_{\partial W}\left({\mathcal D}_{\partial W} \right)$ as follows:
\begin{align}
  dv_g = \sqrt{g}(r, \xi)\,dr \, dA \qquad (\xi \in \partial W),
\end{align}
where $r={\rm dist}\,(W,*)$. 

We shall use the transplantation method as follows: first, for $(t,v) \in [0,\infty) \times U \mathcal{N}^+(\partial W)$ satisfying $tv \in \overline{{\mathcal B}(\partial W, R)} \cap \overline{{\mathcal D}_{\partial W}}$, define a function $H_R$ on $\overline{B(\partial W,R)}$ by
\begin{align*}
  H_R(\exp_{\partial W}(tv)) = h_1(t),
\end{align*}
where $h_1$ is the function defined by $(8)$. 
Next, using this function $H_R$, define a function $F_R$ on $M$ by 
\begin{align*}
F_R(y)=
\begin{cases}
\ \ h_1(0),  \quad & {\rm if}\quad y\in \overline{W} \\
\ \ H_R(y), \quad & {\rm if}\quad r(y) \in (0, R], \\ 
\ \ 0, \quad & {\rm otherwise}. \\ 
\end{cases}
\end{align*}
Then $F = F_R \in W^{1,2}_c (W\cup \overline{B(\partial W,R)})$, and we get 
\begin{align}
 & \int_{W\cup \overline{B(\partial W,R)}} |\nabla F|^2 dv _g = 
   \int_{\overline{B(\partial W,R)}} |\nabla F|^2 dv _g \nonumber \\
=& \int_{\partial W} dA(\xi) 
   \int_0^{\rho\left( \overrightarrow{n}(\xi) \right) \wedge R} 
   |h_1'(r)|^2 \sqrt{g}(r, \xi)\,dr
\end{align}
and
\begin{align}
 & \int_{W\cup \overline{B(\partial W,R)}} |F|^2 dv _g \nonumber \\
 = 
 & |h_1(0)|^2 \cdot {\rm Vol}(W) + 
   \int_{\partial W} dA(\xi) 
   \int_0^{\rho\left( \overrightarrow{n}(\xi) \right) \wedge R} |h_1(r)|^2 
   \sqrt{g}(r, \xi)\,dr,
\end{align}
where $\rho\left( \overrightarrow{n}(\xi) \right) \wedge R = \min \{ \rho\left( \overrightarrow{n}(\xi) \right), R \}$. 

Now, for each $\xi \in \partial W$, 
\begin{align}
 & \int_0^{\rho\left( \overrightarrow{n}(\xi) \right) \wedge R} 
   |h_1'|^2(r) \sqrt{g}(r, \xi)\,dr \nonumber \\
=& \Bigl[ h_1(r) h_1'(r) \sqrt{g}(r,\xi) \Bigr]
   _{r=0}^{r=\rho\left( \overrightarrow{n}(\xi) \right) \wedge R} 
   - \int_0^{\rho\left( \overrightarrow{n}(\xi) \right) \wedge R} h_1(r)
   \bigl\{ h_1'(r) \sqrt{g}(r,\xi)\bigr\}' \,dr \nonumber \\
=& (h_1 h_1') \bigl( \rho\left( \overrightarrow{n}(\xi) \right) \wedge R \bigr)
   \cdot \sqrt{g}\left( \rho\left( \overrightarrow{n}(\xi) \right) \wedge R, 
   \xi \right)
   - \int_0^{\rho\left( \overrightarrow{n}(\xi) \right) \wedge R} 
   h_1 \bigl\{ h_1'\sqrt{g}(r,\xi) \bigr\}' \,dr \nonumber \\
\le 
 & - \int_0^{\rho\left( \overrightarrow{n}(\xi) \right) \wedge R} 
   h_1(r) \bigl\{ h_1'(r)\sqrt{g}(r,\xi) \bigr\}' \,dr \nonumber \\
=& - \int_0^{\rho\left( \overrightarrow{n}(\xi) \right) \wedge R} 
   h_1(r) \left\{ h_1''(r) 
   + \frac{\partial _r \sqrt{g}(r,\xi)}{\sqrt{g}(r,\xi)} 
   h_1'(r) \right\} \sqrt{g}(r,\xi) \, dr \nonumber \\
\le 
 & - \int_0^{\rho\left( \overrightarrow{n}(\xi) \right) \wedge R} 
   h_1(r) \left\{ h_1''(r) + (n-1) S(r) h_1'(r) \right\} 
   \sqrt{g}(r,\xi) \, dr \nonumber \\
=& \lambda _D \bigl( B(2kR)_{M_{\rm model}} \bigr) 
   \int_0^{\rho\left( \overrightarrow{n}(\xi) \right) \wedge R} 
   |h_1|^2(r) \sqrt{g}(r,\xi) \, dr ,
\end{align}
where we have used the fact $h_1'(0)=0$ at the first equality; we have used $(10)$ and $(11)$ at the first inequality; we have used $(10)$, $(11)$, $\Delta r =\frac{\partial _r \sqrt{g}}{\sqrt{g}}$, and $(2)$ at the second inequality; we have used $(9)$ at the last equality. 

Integrating both side of the inequality $(17)$ over $\partial W$ and combining $(15)$ and $(16)$, we see that
\begin{align*}
 & \int_{\overline{B(\partial W,R)}} |\nabla F|^2 dv_g \nonumber \\
=& \int_{\partial W}\,dA(\xi) \int_0^{\rho(\overrightarrow{n}(\xi))\wedge R} 
   |h_1'(r)|^2 \sqrt{g}(r,\xi) \,dr \nonumber \\
 \le 
 & \lambda_D \big{(} B(R)_{\widehat{M}_{{\rm model}}} \big{)} 
   \left\{ \int_{W\cup \overline{B(\partial W,R)}} |F|^2 dv_g 
   - |h_1 (0)|^2 \cdot {\rm Vol}(W) \right\}.
\end{align*}
Hence, we have
\begin{align}
   \frac{\int_M |\nabla F_R|^2\,dv_g}{\int_M |F_R|^2\,dv_g} 
 \le 
 & \lambda_D\left( B(R)_{\widehat{M}_{{\rm model}}}\right) 
   \left\{ 
   1 - \frac{|h_1 (0)|^2 \cdot {\rm Vol}(W)}{\int_M |F_R|^2\,dv_g} 
   \right\} \nonumber \\
 < 
 & \lambda_D\left( B(R)_{\widehat{M}_{{\rm model}}}\right) .
\end{align}
This inequality $(18)$ holds for all $R \ge R_0(n,\beta,\kappa,\delta)$, and hence, setting $R_i=R_0(n,\beta,\kappa,\delta)+i$ and considering the corresponding functions $F_{R_i}$ as above, we get the sequence $\{F_{R_i}\}$ of functions in $W^{1,2}_c(M)$ satisfying
\begin{align*}
 & \frac{\int_M |\nabla F_{R_i}|^2 dv_g}
   {\int_M |F_{R_i}|^2 dv_g} <  \frac{(n-1)^2\kappa}{4} ;\\
 & {\rm supp} \, F_{R_i}= \overline{B(\partial W,R_i)}.
\end{align*}
Since $\{F_{R_i}\}_{i=1}^{\infty}$ spans the infinite dimensional subspace in $W^{1,2}_c(M)$, we obtain the conclusion of Theorem $1.1$ by mini-max principle. 

Taking $W=\{ y \in M \mid {\rm dist}\,(y,p_0) < \varepsilon \}$ for $0 < \varepsilon <\min \{ {\rm inj}\,(p_0), R_0 \}$ in Theorem $1.1$, we get Theorem $1.2$, where ${\rm inj}\,(p_0)$ stands for the injectivity radius at $p_0$. 
\section{Proof of Proposition $1.1$} 

In order to prove Proposition $1.1$, we first quote the following theorem from \cite{A-K}:
\begin{thm}
Let $(M, g)$ be a complete noncompact Riemannian $n$-manifold, where $n \ge 2$. Assume that one of ends of $M$, denoted by $E$, has a compact connected $C^{\infty}$ boundary $W := \partial E$ such that the outward normal exponential map $\exp_W : \mathcal{N}^+(W) \rightarrow E$ is a diffeomorphism, where 
\begin{align*} 
  \mathcal{N}^+(W) 
:= 
  \big\{ v \in TM|_W \ \big{|}\ v \ {\rm is\ outward\ normal\ to}\ W \big\}.
\end{align*} 
Assume also that the mean curvature $H_W$ of $W$ with respect to the inward unit normal vector is positive. 
Take a positive constant $R > 0$ satisfying 
\begin{align*}
  H_W \ge \frac{1}{R}\qquad {\rm on}\quad W,
\end{align*}
and set 
\begin{align*}
  \rho(x) := {\rm dist}_g(x, W), \quad 
  \widehat{r}\,(x) := \rho(x) + R \qquad {\rm for}\quad x \in E.
\end{align*}  
Then, for all $u \in C_0^{\infty}(M)$, we have 
\begin{align} 
 & \int_E |\nabla u|^2 \,dv_g \notag \\ 
\ge 
 & \int_E \left\{ \frac{1}{4 \,\widehat{r}\,^2} 
   + \frac{1}{4}(\Delta \,\widehat{r}\,)^2 
   - \frac{1}{2}|\nabla d \,\widehat{r}\,|^2 
   - \frac{1}{2} {\rm Ric}_g(\nabla \,\widehat{r}\,, \nabla \,\widehat{r}\,) \right\} 
   u^2 \,dv_g \notag \\
 & + \frac{1}{2} \int_W \Big( \Delta \,\widehat{r}\, - \frac{1}{R} \Big) u^2 
   \,d\sigma_g \notag \\ 
\ge 
 & \int_E \Big\{ \frac{1}{4 \,\widehat{r}\,^2} 
   + \frac{1}{4}(\Delta \,\widehat{r}\,)^2 
   - \frac{1}{2}|\nabla d \,\widehat{r}\,|^2 
   - \frac{1}{2} {\rm Ric}_g(\nabla \,\widehat{r}\,, \nabla \,\widehat{r}\,) 
   \Big\} |u|^2 dv_g,
\end{align}
where $d\sigma_g$ denote the $(n-1)$-dimensional Riemannian volume measure of $(W, g|_W)$. 
In particular, if $(M, g)$ has a pole $p_0 \in M$, then 
\begin{align*}
  \int_M |\nabla u|^2 \,dv_g 
\ge 
  \int_M \Big{\{} \frac{1}{4r^2} + \frac{1}{4}(\Delta_g r)^2 
  - \frac{1}{2}|\nabla dr|^2 
  - \frac{1}{2} {\rm Ric}_g(\nabla r, \nabla r) \Big{\}} |u|^2 \,dv_g, 
\end{align*} 
where $r(x) := {\rm dist}_g(x, p_0)$ for $x \in M$. 
Recall that a point $p_0$ of a Riemannian manifold $(M, g)$ is called a pole if the exponential map $\exp_{p_0}:T_{p_0}M \to M$ at $p_0$ is a diffeomorphism. 
\end{thm} 

In view of Theorem $1.1$, it suffices to prove the following: if $\beta < 1/(n-1)^2$, $\sigma_{\rm disc}( -\Delta ) \cap \left[ 0, \frac{(n-1)^2 \kappa }{4} \right)$ is finite. 

Let us set $A(r)=\frac{h'(r)}{h(r)}$. 
Then, $A(r)$ satisfies the following Ricatti equation
\begin{align*}
  A'(r) + A^2(r) + K(r) = 0 \quad {\rm on}~~(0,\infty).
\end{align*} 
Assume that $K(r)$ satisfies
\begin{align}
  K(r) \ge 0 \quad {\rm on}~~(0,\infty)
\end{align}
and
\begin{align}
  K(r) = - \kappa + \frac{\beta}{r^2} \qquad \mbox{for}~~ r \ge R_0,
\end{align} 
where $\kappa >0$, $\beta < 1/(n-1)^2$, and $R_0>0$ are constants. 
In view of $(16)$, the comparison theorem implies that
\begin{align}
  A(r) \ge \frac{1}{r} > 0 \quad {\rm on}~~(0,\infty).
\end{align}
Using the comparison theorem again together with $(17)$ and $(18)$ makes
\begin{align}
  A(r) = \sqrt{\kappa} -  \frac{\beta}{2\sqrt{\kappa}\,r^2} 
  + O \left( \frac{1}{r^3} \right) \quad {\rm as}~~r \to \infty.
\end{align}
Therefore, we have
\begin{align*}
  \frac{1}{4}(\Delta r )^2 - \frac{1}{2}|\nabla dr|^2 
=& \frac{1}{4}(n-1)^2A^2(r) - \frac{1}{2} (n-1)A^2(r) \\
=& \frac{(n-1)(n-3)}{4}A^2(r) \\
=& \frac{(n-1)(n-3)}{4} \left( \kappa - \frac{\beta}{r^2} \right) 
   + O \left( \frac{1}{r^3} \right),
\end{align*}
and hence, 
\begin{align*}
 & \frac{1}{4r^2} + \frac{1}{4}(\Delta r )^2 - \frac{1}{2}|\nabla dr|^2 
   - \frac{1}{2} {\rm Ric}_g(\nabla r, \nabla r) \\
=& \frac{1}{4r^2} + \frac{(n-1)(n-3)}{4} \left( \kappa - \frac{\beta}{r^2} 
   \right) - \frac{(n-1)}{2}\left( - \kappa + \frac{\beta}{r^2} \right) 
   + O \left( \frac{1}{r^3} \right) \\
=& \frac{(n-1)^2\kappa}{4} + \frac{1}{4r^2}\left\{ 1 - (n-1)^2\beta \right\} 
   + O \left( \frac{1}{r^3} \right).
\end{align*}
Hence, substituting
\begin{align*}
 E={\bf R}^n - B_0(R), \quad \rho(x)={\rm dist}_g(x,\partial B_0(R)), 
 \quad \widehat{r}\,(x)=\rho(x)+R=r(x)
\end{align*}
into the equation $(15)$ in Theorem $4.1$, we see that the following inequality holds for all $u \in C_0^{\infty}( {\bf R}^n )$ and $R>0$, since the metric is rotationally symmetric: 
\begin{align*} 
 & \int_{{\bf R}^n - B_0(R)} |\nabla u|^2 \,dv_g \\ 
\ge 
 & \int_{{\bf R}^n - B_0(R)} \left\{ \frac{1}{4 r^2} 
   + \frac{1}{4}(\Delta r)^2 
   - \frac{1}{2}|\nabla d r|^2 
   - \frac{1}{2} {\rm Ric}_g(\nabla r, \nabla r,) \right\} 
   |u|^2 \,dv_g \\
 & + \frac{1}{2} \int_{\partial B_0(R)} \Big( \Delta r - \frac{1}{R} \Big) 
   |u|^2 \,d\sigma_g \\
=& \int_{{\bf R}^n - B_0(R)} \left\{ 
   \frac{(n-1)^2\kappa}{4} + \frac{1}{4r^2}\left\{ 1 - (n-1)^2\beta \right\} 
   + O \left( \frac{1}{r^3} \right) \right\} |u|^2 \,dv_g \\
 & + \frac{1}{2} \int_{\partial B_0(R)} 
   \left\{ (n-1)\sqrt{\kappa} - \frac{1}{R} 
   - O\left(\frac{1}{R^2}\right) \right\} |u|^2 \,d\sigma_g,
\end{align*}
where we have used $\Delta r = (n-1)A(r)=(n-1)\sqrt{\kappa}-O(r^{-2})$ (see $(19)$); also, we set $B_0(R)=\{ x \in {\bf R}^n \mid {\rm dist}\,(x,0) =R \}$ and $0$ represents the origin of ${\bf R}^n$. 
Therefore, since $1>(n-1)^2\beta$, there exits a constant $R_1>R_0$ such that 
\begin{align}
 & \int_{{\bf R}^n - B_0(R)} |\nabla u|^2 \,dv_g \nonumber \\ 
\ge 
 & \int_{{\bf R}^n - B_0(R)} \frac{(n-1)^2\kappa}{4}  |u|^2 \,dv_g 
   \qquad {\rm for~all}~u\in C_0^{\infty}({\bf R}^n)~{\rm and}~R\ge R_1.
\end{align}
Now, let $\Delta_{B_0(R_1)}$ be the Laplacian on $(B_0(R_1),dr^2 + h^2(r) g_{S^{n-1}(1)})$ with vanishing Neumann boundary condition and 
\begin{align}
  0=\mu_0 < \mu_1 \le \cdots \le \mu_i \le \mu_{i+1} \le \cdots \nearrow \infty
\end{align}
be its eigenvalues with each eigenvalues repeated according to its multiplicity. 
Also, let $\Delta_{{\bf R}^n - B_0(R_1)}$ be the Laplacian on $\left( {\bf R}^n - B_0(R_1), dr^2 + h^2(r) g_{S^{n-1}(1)} \right)$ with vanishing Neumann boundary condition. 
Then, from $(20)$, we see that 
\begin{align}
  \sigma \left( -\Delta_{{\bf R}^n - B_0(R_1)} \right) \subset 
  \left[\frac{(n-1)^2\kappa}{4},\infty \right). 
\end{align}
Hence, the domain monotonicity principle (vanishing Neumann boundary data) due to Courant-Hilbert \cite{C-H} (see also \cite{Cha} pp.~13), together with $(21)$ and $(22)$, implies that 
\begin{align*}
 & \sharp \left\{ \lambda \in \sigma_{{\rm disc}}(-\Delta) ~\Big|~ \lambda 
  < \frac{(n-1)^2\kappa}{4} \right\} \\
\le 
 & \sharp \left\{ \mu_i ~\Big|~ \mu_i < \frac{(n-1)^2\kappa}{4} \right\} 
< \infty.
\end{align*}
Here, ``$\,\sharp\,$'' represents the counting function of eigenvalues with each eigenvalues repeated according to its multiplicity. 
Thus, we have proved Proposition $1.1$.

\section{Applications and remarks}

Reflecting our proof, we see that the following holds:
\begin{prop}
Let $W$ be a relatively compact open subset of a Riemannian manifold $(M,g)$ of dimension $n$. 
Assume that $\partial W$ is $C^{\infty}$, and that the outward normal exponential map $\exp_{\partial W}:N^+(\partial W) \to M\backslash W$ is a diffeomorphism. 
Moreover, assume that $\Delta r = (n-1) \left\{ \sqrt{\kappa}-\frac{\beta}{2\sqrt{\kappa}\,r^2} + O \left( \frac{1}{r^3} \right)\right\}~~(r\to \infty)$ on $M\backslash W$, where $r={\rm dist}(W,*)$ on $M\backslash W$$;$ $\kappa$ and $\beta$ are positive constants. 
If $\beta > 1/(n-1)^2$, then $\sigma_{{\rm ess}}(-\Delta)=\left[ \frac{(n-1)^2\kappa}{4}, \infty \right)$ and $\sigma_{{\rm disc}}(-\Delta)$ is infinite. 
\end{prop}

In Proposition $5.1$, $\partial W$ may have a finite number of components. 
Using Proposition $5.1$, we can construct examples with infinite number of the discrete spectrum of the Laplacian. 

In Theorem $1.1$ and $1.2$, we assumed that 
\begin{align}
  \min \sigma_{{\rm ess}} (-\Delta ) = \frac{(n-1)^2 \kappa }{4}.
\end{align}
The condition $(27)$ is satisfied if the inequality
\begin{align}
  \sup \{ \mathfrak{h}(M\backslash K) \mid K \subset M{\rm ~is~compact} \} 
  \ge 
  (n-1)\sqrt{\kappa}
\end{align}
holds under our curvature assumption, where $\mathfrak{h}(M\backslash K):=\inf\left\{ \frac{{\rm Vol}_{n-1}(\partial \Omega )}{{\rm Vol}_{n}(\Omega)} \mid \Omega \subset M\backslash K \right\}$ is the Cheeger constant of $M\backslash K$; 
next, the condition $(28)$ holds if there exists a $C^{\infty}$-function $f$ defined near infinity satisfying 
\begin{align*}
  \liminf_{M \ni y \to \infty}\Delta f (y)\ge (n-1)\sqrt{\kappa} ~~
  {\rm and}~~ 
  |\nabla f|\le 1.
\end{align*}

Modifying our arguments, we also get the following:
%
%
\begin{thm}
Let $(M,g)$ be an $n$-dimensional noncompact complete Riemannian manifold and $W$ a relatively compact open subset of $M$ with $C^{\infty}$-boundary $\partial W$. 
We set $r(*):={\rm dist}(*,\partial W)$ on $M\backslash W$. 
Let $\exp_{\partial W}: \mathcal{N}^+(\partial W) \rightarrow M\backslash W$ be the outward normal exponential map and ${\mathcal Cut}(\partial W)$ the corresponding cut locus of $\partial W$ in $M \backslash W$, where 
\begin{align*} 
  \mathcal{N}^+(\partial W) 
:= 
  \big\{ v \in TM|_{\partial W} \, \big{|} \, v \ {\rm is\ outward\ normal\ to}\ \partial W \big\}.
\end{align*}
Assume that there exist positive constants $\kappa$ and $R_0$ and positive-valued continuous function $\varphi$ of $t \in [r_0,\infty)$ such that
\begin{align*}
 & {\rm Ric}_{g}\left( \nabla r, \nabla r \right) (y)
   \ge 
   - (n-1)\kappa - \varphi (r(y)) \\
 & \hspace{30mm}\mbox{for}~y \in M \backslash 
   \left( W \cup {\mathcal Cut}(\partial W) \right)~\mbox{with}~ r(y) \ge R_0
\end{align*}
and 
$$
  \lim_{t\to\infty} \varphi(t)=0.
$$
Then, $\sigma_{\rm ess}( -\Delta ) \cap \left[ 0, (n-1)^2 \kappa /4 \right] \neq \emptyset$, where $\sigma_{\rm ess}( -\Delta )$ stands for the essential spectrum of $-\Delta$. 
\end{thm}

Theorem $5.1$ immediately implies the following

\begin{cor}
Let $(M,g)$ be an $n$-dimensional noncompact complete Riemannian manifold and $p_0$ a fixed point of $M$. 
We set $r(*):={\rm dist}(*,p_0)$ and denote by ${\mathcal Cut}(p_0)$ the cut locus of $p_0$. 
Assume that there exist positive constants $\kappa$ and $R_0$ and positive-valued continuous function $\varphi$ of $t \in [r_0,\infty)$ such that
\begin{align*}
 & {\rm Ric}_{g}\left( \nabla r, \nabla r \right) (y)
   \ge 
   - (n-1)\kappa - \varphi (r(y)) \\
 & \hspace{30mm}\mbox{for}~y \in M \backslash 
   \left( W \cup {\mathcal Cut}(p_0) \right)~\mbox{with}~ r(y) \ge R_0
\end{align*}
and 
$$
 \lim_{t\to\infty} \varphi(t)=0.
$$
Then, 
$\sigma_{\rm ess}( -\Delta ) \cap \left[ 0, (n-1)^2 \kappa /4 \right] \neq \emptyset$. 
\end{cor}
Corollary $5.1$ is a generalization of one of Donnelly's theorems \cite{D} which asserts that $\sigma_{\rm ess}( -\Delta ) \cap \left[ 0, (n-1)^2 \kappa /4 \right] \neq \emptyset$ under assumption that ${\rm Ric}_{g} \ge -(n-1)\kappa$ on all of $(M,g)$. 


\vspace{1mm}
\begin{flushleft}
Hironori Kumura\\ 
Department of Mathematics\\ 
Shizuoka University\\ 
Ohya, Shizuoka 422-8529\\ 
Japan\\
E-mail address: smhkumu@ipc.shizuoka.ac.jp
\end{flushleft}

\end{document}